\documentclass[12pt,a4paper,oneside]{amsart}
\usepackage[a4paper]{geometry}
\usepackage{fouriernc}
\DeclareMathAlphabet{\mathcal}{OMS}{cmsy}{m}{n} 

\usepackage{amssymb,enumerate,mathrsfs,perpage}

\newtheorem{theorem}{Theorem}
\newtheorem{proposition}{Proposition}
\newtheorem{lemma}{Lemma}

\newcommand{\vertbar}{\>|\>}
\newcommand{\set}[2]{\ensuremath{\{\> #1 \vertbar #2 \>\}}}

\MakePerPage[2]{footnote}

\AtBeginDocument
{
  \setlength{\leftmargini}{1.25em}
  \setlength{\jot}{0.1pt}  
}

\DeclareMathOperator{\ad}{ad}
\DeclareMathOperator{\Aut}{Aut}
\DeclareMathOperator{\Diag}{Diag}
\DeclareMathOperator{\Exp}{Exp}
\DeclareMathOperator{\GF}{\mathsf{GF}}

\DeclareMathOperator{\id}{id}
\DeclareMathOperator{\ess}{\mathsf{S}}

\begin{document}

\title{On simple $15$-dimensional Lie algebras in characteristic $2$}

\author{Alexander Grishkov}
\address[Alexander Grishkov]{University of S\~ao Paulo, Brazil and 
Omsk F.M. Dostoevsky State University, Russia}
\email{shuragri@gmail.com}

\author{Henrique Guzzo Jr.}
\address[Henrique Guzzo Jr.]{University of S\~ao Paulo, Brazil}
\email{guzzo@ime.usp.br}

\author{Marina Rasskazova}
\address[Marina Rasskazova]{Omsk State Technical University, Russia}
\email{marinarasskazova@yandex.ru}

\author{Pasha Zusmanovich}
\address[Pasha Zusmanovich]{University of Ostrava, Czech Republic}
\email{pasha.zusmanovich@gmail.com}

\date{March 26, 2021}

\begin{abstract}
Motivated by the recent progress towards classification of simple 
finite-dimensional Lie algebras over an algebraically closed field of 
characteristic $2$, we investigate such $15$-dimensional algebras.
\end{abstract}

\maketitle

\section*{Introduction}

The classification of finite-dimensional simple Lie algebras over an 
algebraically closed field has a long and interesting history. The 
characteristic zero case, due to Killing, Cartan, and Dynkin, is nowadays a 
classic. The case of characteristic $p>3$ was accomplished relatively 
recently, due to the efforts of dozens of people, spread over more than 50 years
and hundreds of papers, and culminated in the three volumes of Strade, 
\cite{strade}. The cases of characteristics $2$ and $3$ remain open, although a
lot of efforts were done recently to augment the classification program by 
small characteristics specifics, in particular, to put Lie superalgebras into 
play (see, for example, a monumental treatise by Leites and his collaborators, 
\cite{leites-and-co}, and references therein).

In \cite{GZ}, we did another small step towards the classification of 
finite-dimensional simple Lie algebras over an algebraically closed field of 
characteristic $2$: it was proved that any such algebra of absolute toral tank 
$2$, and having a Cartan subalgebra of toral rank one, is $3$-dimensional. In 
this, we based on the paper of Skryabin, \cite{skryabin}, where it was proved, 
among other things, that in characteristic $2$ there are no simple Lie algebras
of absolute toral rank $1$, and simple Lie algebras having a Cartan subalgebra 
of toral rank $1$ where characterized as certain filtered deformations of semisimple Lie 
algebras with the socle of the form $S \otimes \mathcal O$, where $S$ is a 
simple Lie algebra either of Zassenhaus or Hamiltonian type, and $\mathcal O$ is
a reduced polynomial algebra.

In the process of the proof of the main result of \cite{GZ}, we have constructed
a $2$-parameter family $\mathscr L(\beta,\delta)$ of $15$-dimensional simple 
Lie algebras (\cite[\S 6]{GZ}). The algebra $\mathscr L(0,0)$ in this family 
coincides with the smallest algebra in the series constructed by Skryabin 
(\cite[Example at pp. 691--692]{skryabin}).

It is the purpose of this paper to study the family $\mathscr L(\beta,\delta)$.
Among other things, we prove that all algebras within the family are isomorphic
to the same algebra $\mathscr L$ (\S \ref{sec-2param}), and we determine the 
absolute toral rank (\S \ref{sec-subalg}), and the automorphism group of 
$\mathscr L$ (\S \ref{sec-aut}). In passing, we introduce the notion of a thin 
decomposition of a simple Lie algebra with respect to a torus (\S \ref{ss-thin})
which, we suggest, should play a role in classification efforts (see open 
questions in \S \ref{s-quest}).

Throughout the paper, the ground field $K$ is assumed to be perfect of 
characteristic $2$, unless stated otherwise. Although the initial problem 
assumes the ground field is algebraically closed, for our purposes it will be 
enough to assume that square roots exist in $K$. This allows to include in our
consideration the case $K = \GF(2)$, what can be useful in some circumstances
(cf. \S \ref{sec-subalg} and \S \ref{ss-eick}).

Our terminology and notation is mostly standard. $C_L(X)$ denotes the 
centralizer of a set $X$ in a Lie algebra $L$; the linear span of a set $X$ is 
denoted either as $KX$ or as $\langle X \rangle$; the $2$-envelope of a Lie 
algebra $L$ is denoted by $L_2$; $\id_L$ denotes the identity map on $L$. Other
notation is explained as soon as it is introduced in the text.

\section{The $2$-parameter family $\mathscr L(\beta,\delta)$}\label{sec-2param}

Recall the definition of the family $\mathscr L(\beta,\delta)$, depending of two
parameters $\beta, \delta \in K$, of $15$-di\-men\-si\-o\-nal simple Lie 
algebras constructed in \cite[\S 6]{GZ}. These algebras are filtered 
deformations of the semisimple Lie algebra of the form
\begin{equation}\label{eq-ss}
\ess \otimes \mathcal O_1(2) + g \otimes \langle 1,x \rangle + \partial , 
\end{equation}
where $\ess$ is the $3$-dimensional simple Lie algebra with the basis 
$\{e,f,h\}$ and multiplication table
\begin{equation}\label{eq-mult}
[e,h] = e, \quad [f,h] = f, \quad [e,f] = h ,
\end{equation}
$g = (\ad f)^2$ is an outer derivation of $\ess$, $\mathcal O_1(2)$ is the 
$4$-dimensional divided power algebra with the basis $\{1,x,x^{(2)},x^{(3)}\}$,
and $\partial: x^{(i)} \to x^{(i-1)}$ is the special derivation of 
$\mathcal O_1(2)$.

For our purpose, it will be convenient to relabel the basis elements of the
algebras $\mathscr L(\beta,\delta)$ as follows:
\begin{alignat*}{7}
b_1 &= e \otimes 1 \qquad  &c_1 &= e \otimes x^{(3)}  \\
b_2 &= f \otimes 1 \qquad  &c_2 &= f \otimes x^{(3)}  \\
b_3 &= h \otimes 1 \qquad  &c_3 &= h \otimes x^{(3)}  \\
b_4 &= e \otimes x \qquad  &c_4 &= g \otimes 1         \\
b_5 &= f \otimes x \qquad  &c_5 &= g \otimes x         \\
b_6 &= h \otimes x \qquad  &d &= \partial              \\
b_7 &= e \otimes x^{(2)}   \\
b_8 &= f \otimes x^{(2)}   \\
b_9 &= h \otimes x^{(2)} 
\end{alignat*}

In terms of this basis, the multiplication table of $\mathscr L(\beta,\delta)$
(see \cite[(5.4)]{GZ}) reads:

\begin{equation}\label{eq-m}
\begin{array}{|c||c|c|c|c|c|c|c|c|c|c|c|c|c|c|} \hline
&b_2&b_3&b_4&b_5&b_6&b_7&b_8&b_9&c_1&c_2&c_3&c_4&c_5&d \\ \hline\hline
b_1&b_3&b_1&\beta c_5&b_6&b_4&\delta c_4&b_9&b_7&\delta c_5 + d&c_3&c_1&b_2&b_5&\beta c_2 
\\ \hline
b_2&&b_2&b_6&0&b_5&b_9&0&b_8&c_3&0&c_2&0&0&0            \\ \hline
b_3&&&b_4&b_5&0&b_7&b_8&0&c_1&c_2&0&0&0&0               \\ \hline
b_4&&&&0&0&\delta c_5 + d&c_3&c_1&b_3&c_4&b_2&b_5&0&b_1 \\ \hline
b_5&&&&&0&c_3&0&c_2&c_4&0&0&0&0&b_2                     \\ \hline
b_6&&&&&&c_1&c_2&0&0&0&c_4&0&0&b_3                      \\ \hline
b_7&&&&&&&0&0&b_6&c_5&b_5&b_8&c_2&b_4                   \\ \hline
b_8&&&&&&&&0&c_5&0&0&0&0&b_5                            \\ \hline
b_9&&&&&&&&&0&0&c_5&0&0&b_6                             \\ \hline
c_1&&&&&&&&&&0&b_8&c_2&0&b_7                            \\ \hline
c_2&&&&&&&&&&&0&0&0&b_8                                 \\ \hline
c_3&&&&&&&&&&&&0&0&b_9                                  \\ \hline
c_4&&&&&&&&&&&&&0&0                                     \\ \hline
c_5&&&&&&&&&&&&&&c_4                                    \\ \hline
\end{array}
\end{equation}

\medskip

\begin{lemma}
For any $\beta,\delta \in K$, we have 
$\mathscr L(\beta,\delta) \simeq \mathscr L(\beta,0)$.
\end{lemma}

\begin{proof}
Let us consider a new basis 
$$
\{
b_1^\prime, b_2, b_3^\prime, b_4^\prime, b_5, b_6, b_7^\prime, b_8, b_9^\prime,
c_1^\prime, c_2, c_3, c_4, c_5, d^\prime
\}
$$
of the algebra $\mathscr L(\beta,\delta)$, where
\begin{align*}
b_1^\prime &= b_1 + \sqrt{\delta} b_6 + \delta b_8 \\
b_3^\prime &= b_3 + \sqrt{\delta} b_5              \\
b_4^\prime &= b_4 + \delta c_2                     \\
b_7^\prime &= b_7 + \sqrt{\delta} c_3              \\
b_9^\prime &= b_9 + \sqrt{\delta} c_2              \\
c_1^\prime &= c_1 + \sqrt{\delta} c_4              \\
d^\prime   &= d + \sqrt{\delta} b_2 .
\end{align*}

It is straightforward to check that in this basis the multiplication table 
coincides with the multiplication table (\ref{eq-m}) with $\delta = 0$. 
\end{proof}

\begin{lemma}
For any $\beta \in K$, we have $\mathscr L(\beta,0) \simeq \mathscr L(0,0)$.
\end{lemma}

\begin{proof}
Consider a new basis
$$
\{
b_1^\prime, b_2, b_3^\prime, b_4^\prime, b_5, b_6^\prime, b_7, b_8, b_9, 
c_1^\prime, c_2, c_3, c_4, c_5, d^\prime
\}
$$
of the algebra $\mathscr L(\beta,0)$, where
\begin{align*}
b_1^\prime &= b_1 + \beta^2 b_9  \\
b_3^\prime &= b_3 + \beta^2 b_8  \\
b_4^\prime &= b_4 + \beta^2 c_3  \\
b_6^\prime &= b_6 + \beta^2 c_2  \\
c_1^\prime &= b_1 + \beta^2 c_5  \\
d^\prime   &= d + \beta^2 b_5 .
\end{align*}

It is straightforward to check that in this basis the multiplication table 
coincides with the multiplication table (\ref{eq-m}) with $\beta = \delta = 0$.
\end{proof}

An immediate corollary of these two lemmas is

\begin{theorem}\label{t1}
For any $\beta,\delta \in K$, we have 
$\mathscr L(\beta,\delta) \simeq \mathscr L(0,0)$.
\end{theorem}

Since the algebra $\mathscr L(0,0)$ is isomorphic to the smallest 
$15$-dimensional algebra in the family of simple Lie algebras constructed by 
Skryabin in \cite[pp.691--692]{skryabin}, we will refer to it in the rest of the
paper as the \emph{Skryabin algebra}. The rest of the paper is devoted to
elucidation of some properties, and computation of some invariants of the 
Skryabin algebra.

\section{$2$-envelope, derivations, sandwich subalgebra}\label{sec-misc}

We continue to employ the basis $\{b_1,\dots,b_9,c_1,\dots,c_5,d\}$ of the 
Skryabin algebra $\mathscr L$ from the previous section (referred as the
\emph{standard basis} in what follows). The multiplication table is given by 
(\ref{eq-m}), where $\beta = \delta = 0$.

The Skryabin algebra is not a $2-$algebra; its $2$-envelope $\mathscr L_2$ has 
dimension $19$, with additional basis elements 
$b_1^{[2]}, b_4^{[2]}, b_7^{[2]}, c_3^{[2]}$, the $2$-map:
\begin{gather*}
b_2^{[2]} = c_4 ,\quad
b_3^{[2]} = b_3 ,\quad
c_1^{[2]} = b_9 ,\quad
d^{[2]} = b_4^{[2]} ,
\\
b_5^{[2]} = b_6^{[2]} = b_8^{[2]} = b_9^{[2]} = c_2^{[2]} = c_4^{[2]} = 
c_5^{[2]} = 0 ,
\\
b_7^{[4]} = b_7^{[2]} ,\quad b_1^{[4]} = b_4^{[4]} = c_3^{[4]} = 0 ,
\end{gather*}
and the multiplication: 
\begin{equation*}
\begin{array}[t]{|c||c|c|c|c|c|c|}\hline
    & b_1^{[2]} & b_4^{[2]} & b_7^{[2]} & c_3^{[2]} \\ \hline\hline
b_1 & 0         & 0         & 0         & b_8  \\ \hline
b_2 & b_1       & 0         & 0         & 0    \\ \hline
b_3 & 0         & 0         & 0         & 0    \\ \hline
b_4 & 0         & 0         & b_4       & c_2  \\ \hline
b_5 & b_4       & 0         & b_5       & 0    \\ \hline
b_6 & 0         & 0         & b_6       & 0    \\ \hline
b_7 & 0         & b_1       & 0         & 0    \\ \hline
b_8 & b_7       & b_2       & 0         & 0    \\ \hline
b_9 & 0         & b_3       & 0         & 0    \\ \hline
c_1 & 0         & b_4       & c_1       & 0    \\ \hline
c_2 & c_1       & b_5       & c_2       & 0    \\ \hline
c_3 & d         & b_6       & c_3       & 0    \\ \hline
c_4 & b_3       & 0         & 0         & 0    \\ \hline
c_5 & b_6       & 0         & c_5       & 0    \\ \hline
d   & 0         & 0         & d         & c_5  \\ \hline
\end{array}
\qquad
\begin{array}[t]{|c||c|c|c|} \hline
          & b_4^{[2]} & b_7^{[2]} & c_3^{[2]} \\ \hline\hline
b_1^{[2]} & 0         & 0         & b_9       \\ \hline
b_4^{[2]} &           & 0         & c_4       \\ \hline
b_7^{[2]} &           &           & 0         \\ \hline
\end{array}
\end{equation*}

\bigskip

In what follows, we will employ the following standard notation for arbitrary
elements of $\mathscr L$:
\begin{align}
&\lambda_1 b_1 + \dots + \lambda_9 b_9 + \mu_1 c_1 + \dots + \mu_5 c_5 + \eta d
,
\notag
\intertext{and of $\mathscr L_2$:}
&\lambda_1 b_1 + \dots + \lambda_9 b_9 + \mu_1 c_1 + \dots + \mu_5 c_5 + \eta d
+ \xi_1 b_1^{[2]} + \xi_4 b_4^{[2]} + \xi_7 b_7^{[2]} + \xi_3 c_3^{[2]} ,
\label{eq-genel}
\end{align}
where 
$\lambda_1,\dots,\lambda_9,\mu_1,\dots,\mu_5,\eta,\xi_1,\xi_4,\xi_7,\xi_3 
\in K$.

\begin{proposition}\label{prop-der}
The derivation algebra of the Skryabin algebra coincides with its $2$-envelope.
\end{proposition}

\begin{proof}
Straightforward calculations.
\end{proof}

Recall that an element $x$ of a Lie algebra $L$ is called a \emph{sandwich} if 
$(\ad x)^2 = 0$ and $[[L,x],[L,x]] = 0$. (It is well-known -- and easy to see --
that if the characteristic of the ground field is different from $2$, then the 
second condition follows from the first one, but in characteristic $2$ this is 
not true). The set of all sandwiches is multiplicatively closed, what implies 
that the \emph{sandwich subalgebra}, i.e., the subalgebra of $L$ generated by 
all sandwiches, is just the linear span of sandwiches.

It follows from the result first proved by Kostrikin and Zelmanov, that a Lie 
algebra generated by sandwiches is nilpotent. See \cite[\S 3.2]{vaughan-lee} for
details and further references.

More generally, we will call a derivation $D$ of a Lie algebra $L$ a 
\emph{sandwich derivation}, if $D^2 = 0$ and $[D(L),D(L)]$ = 0.

\begin{lemma}\label{prop-sand}
The set of elements $x \in \mathscr L_2$ such that
\begin{equation}\label{eq-lx}
[[\mathscr L,x],[\mathscr L,x]] = 0
\end{equation}
coincides with the linear span of elements $c_2$, $c_4$, $c_5$, $c_3^{[2]}$.
\end{lemma}

\begin{proof}
It is a matter of straightforward verification that any linear span of elements
$c_2$, $c_4$, $c_5$, $c_3^{[2]}$ satisfies the condition (\ref{eq-lx}).

Conversely, let $x$ be an arbitrary element (\ref{eq-genel}) of $\mathscr L_2$ 
satisfying this condition. We perform the following calculations:
\begin{itemize}

\item
Collecting in the equality $[[x,b_8],[x,c_5]] = 0$ the terms containing $c_1$ 
and $c_2$, we get, respectively, $\xi_1^2 = 0$ and $\lambda_1^2 + \xi_1\xi_7 = 0$,
whence $\xi_1 = 0$ and $\lambda_1 = 0$. 

\item
Collecting in the equality $[[x,b_6],[x,c_4]] = 0$ the terms containing $c_5$, 
we get $\lambda_7^2 = 0$, whence $\lambda_7 = 0$.

\item
Collecting in the equality $[[x,b_1],[x,c_4]] = 0$ the terms containing $b_8$, 
we get $\mu_1^2 = 0$, whence $\mu_1 = 0$. 

\item
Collecting in the equality $[[x,b_6],[x,c_2]] = 0$ the terms containing $b_8$, 
we get $\eta^2 = 0$, whence $\eta = 0$.

\item
Collecting in the equality $[[x,b_8],[x,c_3]] = 0$ the terms containing $b_5$ 
and $c_2$, we get, respectively, $\xi_4^2 = 0$ and 
$\lambda_4^2 + \lambda_3\xi_4 = 0$, whence $\xi_4 = 0$ and $\lambda_4 = 0$.

\item
Collecting in the equality $[[x,b_1],[x,b_8]] = 0$ the terms containing $b_9$ 
and $c_4$, we get, respectively, $\lambda_3^2 = 0$ and $\lambda_6^2 = 0$, whence
$\lambda_3 = 0$ and $\lambda_6 = 0$.

\item
Collecting in the equality $[[x,b_4],[x,c_2]] = 0$ the terms containing $c_4$, 
we get $\xi_7^2 = 0$, whence $\xi_7 = 0$.

\item
Collecting in the equality $[[x,b_1],[x,b_6]] = 0$ the terms containing $b_5$ 
and $c_2$, we get, respectively, $\lambda_2^2 = 0$ and $\mu_3^2 = 0$, whence 
$\lambda_2 = 0$ and $\mu_3 = 0$.

\item
Collecting in the equality $[[x,b_1],[x,c_3]] = 0$ the terms containing $c_2$, 
we get $\lambda_9^2 = 0$, whence $\lambda_9 = 0$.

\item
Collecting in the equality $[[x,b_1],[x,b_4]] = 0$ the terms containing $c_5$, 
we get $\lambda_8^2 = 0$, whence $\lambda_8 = 0$.

\item
Collecting in the equality $[[x,b_1],[x,b_7]] = 0$ the terms containing $c_4$, 
we get $\lambda_5^2 = 0$, whence $\lambda_5 = 0$.

\end{itemize}

We are left with a linear combination of $c_2$, $c_4$, $c_5$, $c_3^{[2]}$.
\end{proof}

As an immediate corollary of this lemma, we have

\begin{proposition}\label{prop-sandr}\hfill
\begin{enumerate}[\upshape(i)]
\item\label{it-sk}
The sandwich subalgebra of $\mathscr L$ is $3$-dimensional abelian, linearly
spanned by elements $c_2$, $c_4$, $c_5$.

\item\label{it-der}
The set of sandwich derivations of $\mathscr L$ forms a $4$-dimensional abelian
subalgebra of $\mathscr L_2$ spanned by the inner derivations $\ad c_2$, 
$\ad c_4$, $\ad c_5$, and by the outer derivation $\ad c_3^{[2]}$.
\end{enumerate}
\end{proposition}

\begin{proof}
(i) By Lemma \ref{prop-sand}, the sandwich subalgebra of $\mathscr L$ lies in
the linear span of $c_2$, $c_4$, $c_5$. Since these elements pairwise commute,
they span the $3$-dimensional abelian subalgebra of $\mathscr L$, and 
$$
(\mu_1 c_2 + \mu_4 c_4 + \mu_5 c_5)^{[2]} = 
\mu_1^2 c_2^{[2]} + \mu_4^2 c_4^{[2]} + \mu_5^2 c_5^{[2]} = 0 ,
$$
thus any element in this $3$-dimensional subalgebra satisfies the condition
$(\ad x)^2 = 0$.

(ii) By Proposition \ref{prop-der}, any derivation is an element of 
$\mathscr L_2$. Apply Lemma \ref{prop-sand} and reason as above.
\end{proof}

Note that from Lemma \ref{prop-sand} and Proposition \ref{prop-sandr} it follows
that in the Skryabin algebra $\mathscr L$, the condition 
$[[\mathscr L,x],[\mathscr L,x]] = 0$ implies $[[\mathscr L,x],x] = 0$. In 
general, this is, of course, not true.

\section{Computations over $\GF(2)$, the absolute toral rank, 
thin decomposition}\label{sec-subalg}

\subsection{Some numerology}\label{ss-num}

In this section we report on computations over $\GF(2)$, performed on the 
computer in GAP, \cite{gap}\footnote{
The GAP code is available at 
\texttt{https://web.osu.cz/$\sim$Zusmanovich/papers/15dim/} .
}.
A brute-force search on the computer shows that in the $2$-envelope of the 
Skryabin algebra $\mathscr L$ over $\GF(2)$ there are 384 toral elements, $6144$
$2$-dimensional tori, 21504 $3$-dimensional tori, 26880 $4$-dimensional tori, 
and no $5$-dimensional tori (no attempt was made to determine their conjugacy 
classes with respect to the automorphism group). 

The centralizer in $\mathscr L$ of each of the $384$ toral elements is 
$7$-dimensional, and for $240$ toral elements the centralizer is (central) 
simple. Among these $240$ simple algebras, $48$ have absolute toral rank $2$,
and $192$ have absolute toral rank $3$. As proved in \cite[Theorem 1]{vl-comp}
(and confirmed by computations in \cite{eick}), over $\GF(2)$ there exist two 
simple $7$-dimensional Lie algebras. These algebras are identified as (forms) of
the Zassenhaus algebra $W_1^{\prime}(3)$, and the Hamiltonian algebra 
$H_2^{\prime\prime}\big((2,1),(1 + x_1^{(3)}x_2)\, dx_1 \wedge dx_2\big)$, 
denoted by us here simply as $W$ and $H$, respectively (see \cite{gg-2010} and 
\cite{gga} for further info, including explicit multiplication tables of these 
algebras, and their identification with some other simple $7$-dimensional Lie 
algebras from the literature). Both algebras have absolute toral rank $3$ over 
an algebraically closed field, but over $\GF(2)$ $W$ has absolute toral rank 
$2$, while $H$ has absolute toral rank $3$; thus the absolute toral rank can be
used to distinguish them as subalgebras of the Skryabin algebra in our 
computations.

\subsection{Some examples}

Let us exhibit explicitly one of the maximal tori, and one of the 
$7$-di\-men\-si\-o\-nal simple subalgebras mentioned in the previous subsection.

Here is just one of the $4$-dimensional tori, linearly spanned by the toral 
elements
\begin{equation}\label{eq-tor}
b_1 + b_3 + b_1^{[2]} , \quad
b_2 + b_3 + c_4 + b_7^{[2]} , \quad 
b_4 + b_6 + b_4^{[2]} + b_7^{[2]} , \quad
b_8 + b_9 + c_1 + c_3 + b_7^{[2]} + c_3^{[2]} .
\end{equation}

Now take the first toral element in this torus, $h = b_1 + b_3 + b_1^{[2]}$. Its
centralizer $C_{\mathscr L}(h)$ has the basis
$$
b_1 + b_3       ,\quad 
b_2 + c_3 + c_4 ,\quad 
b_4 + b_6       ,\quad
b_5 + b_6 + c_5 ,\quad
b_7 + b_9       ,\quad 
c_1 + c_3       ,\quad 
d .
$$

It is straightforward to check that in this basis the multiplication table of
$C_{\mathscr L}(h)$ coincides with the multiplication table of the simple
$7$-dimensional Hamiltonian algebra $H$ (see \cite[\S 3]{gg-2010} or 
\cite[\S 1]{gga}) via the identification
\begin{gather*}
V_0 \leftrightarrow b_5 + b_6 + c_5 ,\quad
V_1 \leftrightarrow c_1 + c_3       ,\quad 
E_1 \leftrightarrow b_7 + b_9       ,\quad
E_0 \leftrightarrow b_2 + c_3 + c_4 ,\quad
\\
F_1 \leftrightarrow b_4 + b_6       ,\quad
F_0 \leftrightarrow d               ,\quad
G   \leftrightarrow b_1 + b_3 .
\end{gather*}

\subsection{The absolute toral rank}

The computer calculation in \S \ref{ss-num} shows that the absolute toral rank 
of the Skryabin algebra over $\GF(2)$ is equal to $4$. However, a bit of 
additional work allows to establish this result over an arbitrary field.

\begin{theorem}
The absolute toral rank of the Skryabin algebra is equal to $4$.
\end{theorem}

\begin{proof}
As we want to establish this result over an arbitrary field $K$, we will 
distinguish the Skryabin algebra 
$\overline{\mathscr L} = \mathscr L \otimes_{\GF(2)} K$ over $K$, and its 
$\GF(2)$-form $\mathscr L$.

A direct computer verification shows that each of the 26880 $4$-dimensional tori
$T$ in $\mathscr L_2$ coincides with its normalizer, i.e., is a Cartan 
subalgebra of $\mathscr L_2$. Consequently, $\overline T = T \otimes_{\GF(2)} K$
is a Cartan subalgebra of $\overline{\mathscr L_2} = (\overline{\mathscr L})_2$.
By \cite[Theorem 2(ii)]{premet-cartan}], $\overline T$ is a torus of the maximal
possible dimension in $\overline{\mathscr L_2}$, and hence the absolute toral 
rank of $\overline{\mathscr L}$ is equal to $4$.

Alternatively, a direct simple proof free from reference to the computer can be
provided by just looking at one of the $4$-dimensional tori, for example, at 
(\ref{eq-tor}). Indeed, a direct calculation shows that (\ref{eq-tor}) coincides
with its normalizer in the $2$-envelope of any Skryabin algebra 
$\overline{\mathscr L}$, and hence by the same reference to \cite{premet-cartan}
is a torus of the maximal possible dimension.
\end{proof}

\subsection{Thin decomposition}\label{ss-thin}

Let us consider the following situation. Assume that a Lie algebra $L$ has
dimension $2^n-1$, the absolute toral rank of $L$ is $n$, $T$ is a torus in the
$2$-envelope of $L$ of (the maximal) dimension $n$, and $T \cap L = 0$. Assume 
further that the roots of the action of $T$ on $L$ are exactly nonzero tuples in
$\GF(2)^n$ (in particular, the centralizer of $T$ in $L$ is zero), and each root space is one-dimensional. Thus, the root space 
decomposition is of the form
\begin{equation}\label{eq-thin}
L = \bigoplus_{\substack{\alpha \in \GF(2)^n \\ \alpha \ne (0,\dots,0)}} 
K e_\alpha
\end{equation}
for some elements $e_\alpha \in L$. In particular, the multiplication table of 
$L$ in the basis $\set{e_{\alpha}}{\alpha \in \GF(2)^n \backslash (0,\dots,0)}$
has the following form: either $[e_\alpha,e_\beta] = e_{\alpha + \beta}$, or 
$[e_\alpha,e_\beta] = 0$.

In such a situation, we will call the decomposition (\ref{eq-thin}) \emph{thin}.
Some simple Lie algebras (for example, the $7$-dimensional algebras, as shown 
in \cite{gg-2010}) admit a thin decomposition. We suggest that this is an 
important property of simple Lie algebras which should be taken into account in
the classification efforts.

A direct computer 
verification shows that the Skryabin algebra admits a thin decomposition with
respect to \emph{each} of the 26880 $4$-dimensional tori in its $2$-envelope. We
will provide explicitly one of them, corresponding to the torus (\ref{eq-tor}).

The corresponding generators of the one-dimensional root spaces are:
\begin{align*}
&e_{0001} = b_2 + b_3 + b_4 + b_6 + c_4                       \\
&e_{0010} = b_2 + b_3 + c_1 + c_3 + c_4                       \\
&e_{0011} = b_2 + b_3 + b_4 + b_6 + c_1 + c_3 + c_4           \\
&e_{0100} = b_1 + b_3 + b_7 + b_9 + d                         \\
&e_{0101} = b_7 + b_9 + d                                     \\
&e_{0110} = b_1 + b_3 + b_5 + b_6 + c_5 + d                   \\
&e_{0111} = b_5 + b_6 + c_5                                   \\
&e_{1000} = b_2 + b_3 + b_8 + b_9 + c_1 + c_2 + c_3           \\
&e_{1001} = b_2 + b_3 + b_4 + b_5 + b_6                       \\
&e_{1010} = b_2 + b_3 + c_1 + c_2 + c_3                       \\
&e_{1011} = b_2 + b_3 + b_4 + b_5 + b_6 + c_1 + c_2 + c_3     \\
&e_{1100} = b_1 + b_2 + b_3 + b_7 + b_8 + b_9 + c_2 + c_3 + d \\
&e_{1101} = b_5 + b_6 + b_7 + b_8 + b_9 + c_2 + c_3 + d       \\
&e_{1110} = b_1 + b_2 + b_3 + b_5 + b_6 + c_2 + c_3 + d       \\
&e_{1111} = b_5 + b_6
\end{align*}
and the corresponding multiplication table of $\mathscr L$ reads:

{\small
\begin{equation*}
\begin{array}{|c||c|c|c|c|c|c|c|c|c|c|c|c|c|c|} \hline
& e_{0010} & e_{0011} & e_{0100} & e_{0101} & e_{0110} & e_{0111} & e_{1000} & 
e_{1001} & e_{1010} & e_{1011} & e_{1100} & e_{1101} & e_{1110} & e_{1111} 
\\ \hline\hline
e_{0001} & e_{0011} & e_{0010} & e_{0101} & e_{0100} & 0 & 0 & e_{1001} & 0 &
e_{1011} & e_{1010} & e_{1101} & e_{1100} &e_{1111} & 0
\\ \hline
e_{0010} && e_{0001} & e_{0110} & e_{0111} & e_{0100} & 0 & 0 & e_{1011} & 0 & 
e_{1001} & e_{1110} & e_{1111} & e_{1100} & 0            
\\ \hline
e_{0011} &&& e_{0111} & e_{0110} & e_{0101} & 0 & e_{1011} & e_{1010} & e_{1001}
& 0 & 0 & e_{1110} & e_{1101} & 0
\\ \hline
e_{0100} &&&& 0 & e_{0010} & e_{0011} & e_{1100} & e_{1101} & e_{1110} & 0 & 
e_{1000} & e_{1001} & 0 & e_{1011} 
\\ \hline
e_{0101} &&&&& e_{0011} & e_{0010} & 0 & e_{1100} & e_{1111} & e_{1110} & 0 & 
e_{1000} & e_{1011} & e_{1010}
\\ \hline
e_{0110} &&&&&& e_{0001} & e_{1110} & e_{1111} & e_{1100} & e_{1101} & 0 & 0 & 
e_{1000} & e_{1001}
\\ \hline
e_{0111} &&&&&&& 0 & 0 & 0 & 0 & e_{1011} & e_{1010} & e_{1001} & 0
\\ \hline
e_{1000} &&&&&&&& e_{0001} & 0 & e_{0011} & e_{0100} & 0 & e_{0110} & 0
\\ \hline
e_{1001} &&&&&&&&& e_{0011} & e_{0010} & e_{0101} & e_{0100} & 0 & 0
\\ \hline
e_{1010} &&&&&&&&&& e_{0001} & e_{0110} & e_{0111} & e_{0100} & 0
\\ \hline
e_{1011} &&&&&&&&&&& e_{0111} & e_{0110} & e_{0101} & 0
\\ \hline
e_{1100} &&&&&&&&&&&& e_{0001} & e_{0010} & e_{0011}
\\ \hline
e_{1101} &&&&&&&&&&&&& 0 & e_{0010}
\\ \hline
e_{1110} &&&&&&&&&&&&&& e_{0001}
\\ \hline
\end{array}
\end{equation*}
}

\section{The automorphism group}\label{sec-aut}

The goal of this section is to determine the automorphism group of the Skryabin
algebra. First we define three types of automorphisms -- exponential 
automorphisms, a certain explicitly defined three one-parameter families, and 
diagonal automorphisms, determine the group generated by them, and then prove 
that they exhaust the whole automorphism group.

\subsection{Exponential automorphisms}

If $D$ is a sandwich derivation of a Lie algebra $L$, then $\exp(D) = 1 + D$ is
an automorphism of $L$, called an \emph{exponential automorphism}. If 
$D = \ad x$ is an inner derivation (i.e., $x$ is a sandwich in $L$), we will use
the shorthand notation $\exp(x)$ instead of $\exp(\ad x)$.

Since $\exp(D)^2 = \exp(2D) = 1$, exponential automorphisms are of order 
$2$, and generate a unipotent subgroup of exponent $2$, denoted by $\Exp(L)$, of
the automorphism group $\Aut(L)$. Note also that if $D_1,D_2 \in L$ are two 
commuting sandwich derivations, then the corresponding automorphisms also 
commute: $\exp(D_1) \circ \exp(D_2) = \exp(D_2) \circ \exp(D_1)$.

According to Proposition \ref{prop-sandr}(\ref{it-der}), the group 
$\Exp(\mathscr L)$ is $4$-dimensional abelian, isomorphic to the additive group
$K^4 = K \oplus K \oplus K \oplus K$. Let us write down its one-parameter 
generators explicitly (here and below we indicate only those basis elements, on
which the automorphism acts non-identically):
$$
\exp (\alpha c_2): 
\begin{array}{l}
b_1 \mapsto b_1 + \alpha c_3  \\ 
b_3 \mapsto b_3 + \alpha c_2  \\
b_4 \mapsto b_4 + \alpha c_4  \\ 
b_7 \mapsto b_7 + \alpha c_5  \\ 
d\>\>\> \mapsto d + \alpha b_8 
\end{array}
$$
$$
\exp (\alpha c_4): 
\begin{array}{l}
b_1 \mapsto b_1 + \alpha b_2  \\ 
b_4 \mapsto b_4 + \alpha b_5  \\ 
b_7 \mapsto b_7 + \alpha b_8  \\ 
c_1 \mapsto c_1 + \alpha c_2  \\
\end{array}
$$
$$
\exp (\alpha c_5): 
\begin{array}{l}
b_1 \mapsto b_1 + \alpha b_5   \\
b_7 \mapsto b_7 + \alpha c_2   \\
d\>\>\> \mapsto d + \alpha c_4
\end{array}
$$
$$
\exp (\alpha c_3^{[2]}): 
\begin{array}{l}
b_1 \mapsto b_1 + \alpha b_8   \\
b_4 \mapsto b_4 + \alpha c_2   \\
d\>\>\> \mapsto d + \alpha c_5 .
\end{array}
$$
Here $\alpha \in K$ is a parameter.

\subsection{On $\exp(\alpha c_4)$}

For a moment, let us return to the realization of $\mathscr L$ as a filtered 
deformation of the semisimple Lie algebra (\ref{eq-ss}). The exponential 
automorphisms $\exp(\alpha c_4)$ are the only automorphisms of $\mathscr L$ 
which are ``lifted'' from automorphisms of the algebra (\ref{eq-ss}). Using the
results about automorphisms of the tensor product of a simple Lie algebra and a
divided power algebra (see, for example, \cite[\S 2.2]{weisf}), it is not 
difficult to describe the automorphism group of (\ref{eq-ss}) (roughly, those 
are automorphisms of $\ess \otimes \mathcal O_1(2)$ invariant under the action 
of $g$ and of $\partial$). Automorphisms which are preserved by the cocycles 
defining the deformation, will be also automorphisms of $\mathscr L$. Among the 
automorphisms of (\ref{eq-ss}), the only automorphisms satisfying this condition, are
automorphisms acting on $\ess \otimes \mathcal O_1(2)$ as 
$\phi(\alpha) \otimes \id_{\mathcal O_1(2)}$, where $\phi(\alpha)$ is an 
automorphism of $\ess$ of the form
$$
e \mapsto e + \alpha f, \quad f \mapsto f, \quad h \mapsto h ,
$$
and leaving $g \otimes \langle 1,x \rangle$ and $\partial$ invariant. In terms
of the standard basis of $\mathscr L$, this is exactly $\exp(\alpha c_4)$. In 
general, the automorphism group of the algebra (\ref{eq-ss}) is much smaller 
than the automorphism group of its deformation $\mathscr L$, which shows that, 
generally, there is no strong relationship between automorphisms of a Lie 
algebra and of its deformation.

\subsection{Automorphisms $\Phi$, $\Psi$, and $\Theta$}

Consider the following three one-parameter families of linear maps on 
$\mathscr L$, depending on the parameter $\alpha \in K$:
\begin{equation*}
\Phi(\alpha):
\begin{array}{l}
b_1 \mapsto b_1 + \alpha b_4                                \\
b_2 \mapsto b_2 + \alpha b_5                                \\
b_7 \mapsto b_7 + \alpha^2 b_2 + \alpha^3 b_5 + \alpha c_1  \\
b_8 \mapsto b_8 + \alpha c_2                                \\
b_9 \mapsto b_9 + \alpha^2 c_4                              \\
c_1 \mapsto c_1 + \alpha^2 b_5                              \\
c_3 \mapsto c_3 + \alpha c_4                                \\
d\>\>\> \mapsto d + \alpha b_3 + \alpha^2 b_6 
\end{array}
\end{equation*}
$$
\Psi(\alpha):
\begin{array}{l}
b_1 \mapsto	b_1 + \alpha b_6 + \alpha c_1 + \alpha^2 c_4 \\
b_2 \mapsto b_2 + \alpha c_2                             \\
b_3 \mapsto b_3 + \alpha b_5                             \\
b_4 \mapsto b_4 + \alpha b_2 + \alpha^2 c_2              \\
b_6	\mapsto b_6 + \alpha c_4                             \\
b_7 \mapsto b_7 + \alpha b_5 + \alpha c_3                \\
b_9 \mapsto b_9 + \alpha c_2 + \alpha c_5                \\
c_1 \mapsto c_1 + \alpha b_8 + \alpha c_4                \\
d\>\>\> \mapsto	d + \alpha b_2 + \alpha b_9 + \alpha^2 c_2 + \alpha^2 c_5
\end{array}
$$

\smallskip

\begin{equation*}
\hspace{10pt}\Theta(\alpha):
\begin{array}{l}
b_1 \mapsto b_1 + \alpha b_7 + \alpha^3 b_8 + \alpha^2 b_9                 \\ 
b_2 \mapsto b_2 + \alpha b_8                                               \\
b_3 \mapsto b_3 + \alpha^2 b_8                                             \\
b_4 \mapsto b_4 + \alpha^2 b_5 + \alpha c_1 + \alpha^2 c_3 + \alpha^3c_5   \\
b_5 \mapsto b_5 + \alpha c_2                                               \\
b_6 \mapsto b_6 + \alpha^2 c_2 + \alpha^2 c_5                              \\
c_1 \mapsto c_1 + \alpha^2 c_2 + \alpha^2 c_5                              \\
c_3 \mapsto c_3 + \alpha c_5                                               \\
d\>\>\> \mapsto d + \alpha^2 b_5 + \alpha b_6 + \alpha^3 c_2 + \alpha^2 c_3 
\end{array}
\end{equation*}

\bigskip

Direct calculations show that all of them are automorphisms of $\mathscr L$, and
\begin{alignat}{3}
&\Phi(\alpha) \>\circ\> \Phi(\alpha^\prime) &\>=\>& \Phi(\alpha + \alpha^\prime)
\notag \\
&\Psi(\alpha) \>\circ\> \Psi(\alpha^\prime) &\>=\>& 
\Psi(\alpha + \alpha^\prime) \>\circ\>
\exp\big(\alpha\alpha^\prime c_3^{[2]}\big)
\label{eq-psi}
\\
&\Theta(\alpha) \>\circ\> \Theta(\alpha^\prime) &\>=\>&
\Theta(\alpha+\alpha^\prime) \>\circ\>
\exp\big((\alpha^2\alpha^\prime + \alpha{\alpha^\prime}^2) c_3^{[2]}\big)
\notag 
\end{alignat}
for any $\alpha, \alpha^\prime \in K$. In particular, the automorphisms
$\Phi(\alpha)$ and $\Theta(\alpha)$ are of order $2$, and the automorphisms 
$\Psi(\alpha)$ are of order $4$.

\subsection{Diagonal automorphisms}\label{ss-diag}

Let $L$ be a Lie algebra with a basis $B$. We will call an automorphism of $L$ \emph{diagonal with respect to $B$} (or just \emph{diagonal} if it is clear
which basis $B$ is meant), if it leaves invariant each one-dimensional subspace
$Kx$, $x \in B$.

\begin{lemma}\label{lemma-diag}
Each diagonal automorphism of $\mathscr L$ with respect to the standard basis 
is of the form
\begin{alignat}{7}\label{eq-diag}
b_1&\mapsto \lambda^{-2} b_1 \qquad &c_1 &\mapsto \lambda c_1   \notag \\
b_2&\mapsto \lambda^2 b_2    \qquad &c_2 &\mapsto \lambda^5 c_2 \notag \\
b_3&\mapsto b_3              \qquad &c_3 &\mapsto \lambda^3 c_3 \notag \\
b_4&\mapsto \lambda^{-1} b_4 \qquad &c_4 &\mapsto \lambda^4 c_4 \notag \\
\Delta(\lambda): \quad 
b_5&\mapsto \lambda^3 b_5    \qquad &c_5 &\mapsto \lambda^5 c_5        \\
b_6&\mapsto \lambda b_6      \qquad &d &\mapsto \lambda^{-1} d  \notag \\
b_7&\mapsto b_7              \notag \\
b_8&\mapsto \lambda^4 b_8    \notag \\
b_9&\mapsto \lambda^2 b_9    \notag
\end{alignat}
where $\lambda \in K^*$.
\end{lemma}

\begin{proof}
Let $x \mapsto \alpha(x)x$, where $x$ is an element in the standard basis, be a 
diagonal automorphism of $\mathscr L$. Denote $\alpha(b_4) = \lambda^{-1}$.

We perform the following calculations:
\begin{itemize}
\item 
$\alpha(b_3) = \alpha(b_1)\alpha(b_2)$, 
$\alpha(b_1) = \alpha(b_1)\alpha(b_3)$, and
$\alpha(b_2) = \alpha(b_2)\alpha(b_3)$
imply $\alpha(b_2) = \alpha(b_1)^{-1}$ and \linebreak $\alpha(b_3) = 1$.

\item
$\alpha(b_1) = \alpha(b_4)\alpha(d)$ and $\alpha(d) = \alpha(b_1)\alpha(c_1)$
imply $\alpha(c_1) = \lambda$.

\item
$\alpha(b_7) = \alpha(c_1)\alpha(d) = \alpha(b_1)\lambda^2$.

\item
$b_7^{[4]} = b_7^{[2]}$ implies $\alpha(b_7)^4 = \alpha(b_7)^2$, thus 
$\alpha(b_7)^2 = 1$ and $\alpha(b_7) = 1$, $\alpha(b_1) = \lambda^{-2}$, 
$\alpha(b_2) = \lambda^2$, and $\alpha(d) = \lambda^{-1}$.

\item
$\alpha(c_3) = \alpha(b_1)\alpha(c_2) = \alpha(c_2)\lambda^{-2}$. 

\item
$\alpha(c_1) = \alpha(b_1)\alpha(c_3)$ implies 
$\lambda = \alpha(c_2)\lambda^{-4}$, thus $\alpha(c_2) = \lambda^5$, and 
$\alpha(c_3) = \lambda^3$.

\item
$\alpha(b_2) = \alpha(b_1)\alpha(c_4)$ implies $\alpha(c_4) = \lambda^4$.
 
\item
$\alpha(b_9) = \alpha(b_2)\alpha(b_7) = \lambda^2$.

\item
$\alpha(b_8) = \alpha(b_2)\alpha(b_9) = \lambda^4$.

\item
$\alpha(c_2) = \alpha(b_7)\alpha(c_5)$ implies $\alpha(c_5) = \lambda^5$.

\item
$\alpha(b_5) = \alpha(b_4)\alpha(c_4) = \lambda^3$.

\item
$\alpha(b_6) = \alpha(b_1)\alpha(b_5) = \lambda$.
\end{itemize}

\end{proof}

Conversely, it is straightforward to verify that each map of the form 
(\ref{eq-diag}) is an automorphism of $\mathscr L$. Therefore, the group of 
diagonal automorphisms $\Diag(\mathscr L)$ is isomorphic to $K^*$, the 
multiplicative group of $K$.

\subsection{Putting all this together}

Denote by $\Aut_0(\mathscr L)$ the group generated by the just defined 
automorphisms of $\mathscr L$.

Direct calculations show that for each $\alpha, \gamma \in K$, any of 
$\Phi(\alpha)$, $\Psi(\alpha)$, and $\Theta(\alpha)$ commutes with any of 
$\exp(\gamma c_2)$, $\exp(\gamma c_4)$, $\exp(\gamma c_5)$, and 
$\exp(\gamma c_3^{[2]})$. Additionally,
\begin{alignat*}{4}
&\Psi(\gamma) &\>\circ\>& \Phi(\alpha) = &\Phi(\alpha) &\>\circ\> \Psi(\gamma) 
\>\circ\> \exp\big(\alpha\gamma c_4\big)
\\
&\Theta(\gamma) &\>\circ\>& \Phi(\alpha) = 
&\Phi(\alpha) &\>\circ\> \Theta(\gamma)
\>\circ\> 
\exp\big(\alpha\gamma^2 c_2\big) \>\circ\> 
\exp\big(\alpha^2\gamma c_4\big) \>\circ\>
\exp\big(\alpha\gamma^2 c_5\big) \>\circ\>
\exp\big(\alpha^2\gamma^2 c_3^{[2]}\big)
\\
&\Theta(\gamma) &\>\circ\>& \Psi(\alpha) = 
&\Psi(\alpha) &\>\circ\> \Theta(\gamma)
\>\circ\> 
\exp\big(\alpha\gamma c_2\big) \>\circ\> 
\exp\big(\alpha\gamma c_5\big) .  
\end{alignat*}

Together with (\ref{eq-psi}) this imply that the group $\mathcal N$ generated by
the exponential automorphisms, and automorphisms $\Phi$, $\Psi$ and $\Theta$, is
a $7$-dimensional unipotent algebraic group. Further, taking into account that 
$\Delta(\lambda)^{-1} = \Delta(\lambda^{-1})$, we have:
\begin{alignat}{5}\label{eq-act}
&\exp(\alpha c_2)^{\Delta(\lambda)} &\>=\>& \exp(\lambda^{-5}\alpha c_2)
\notag \\
&\exp(\alpha c_4)^{\Delta(\lambda)} &\>=\>& \exp(\lambda^{-4}\alpha c_4)
\notag \\
&\exp(\alpha c_5)^{\Delta(\lambda)} &\>=\>& \exp(\lambda^{-5}\alpha c_5) 
\notag \\
&\exp(\alpha c_3^{[2]})^{\Delta(\lambda)} &\>=\>& 
\exp(\lambda^{-6}\alpha c_3^{[2]})
\\
&\Phi(\alpha)^{\Delta(\lambda)} &\>=\>& \Phi(\lambda^{-1}\alpha)
\notag \\
&\Psi(\alpha)^{\Delta(\lambda)} &\>=\>& \Psi(\lambda^{-3}\alpha)
\notag \\
&\Theta(\alpha)^{\Delta(\lambda)} &\>=\>& \Theta(\lambda^{-2}\alpha) .
\notag
\end{alignat}

Therefore, $\mathcal N$ is a normal subgroup in $\Aut_0(\mathscr L)$, and 
$\Aut_0(\mathscr L)$ is isomorphic to the semidirect product 
$K^* \ltimes \mathcal N$, with the action of $K^*$ on $\mathcal N$ defined by 
(\ref{eq-act}).

\subsection{Invariant subspaces}

Now we are going to prove that the automorphisms constructed in the previous 
sections exhaust all automorphisms of $\mathscr L$. To this aim, we determine
certain invariant subspaces in the Skryabin algebra.

\begin{proposition}\label{flag}
The Skryabin algebra $\mathscr L$ possesses the following 
$\Aut(\mathscr L)$-invariant subspaces:
{\setlength{\jot}{4pt}
\begin{gather*}
\langle c_2 \rangle \quad 
\langle c_4 \rangle \quad 
\langle c_5 \rangle 
\\
\langle b_5,c_2 \rangle \quad \langle b_8,c_2 \rangle
\\
V_4 = \langle b_2,b_5,b_8,c_2 \rangle  
\\
V_5 = \langle b_8,c_2,c_3,c_4,c_5 \rangle
\\
V_6 = \langle b_2,b_3,b_5,b_8,c_2,c_4 \rangle
\quad
V_6^\prime = \langle b_5,b_6,b_8,c_2,c_4,c_5 \rangle
\quad
V_6^{\prime\prime} = \langle b_5,b_8,b_9,c_2,c_4,c_5 \rangle
\\
V_7 = \langle b_5,b_6,b_8,b_9,c_2,c_4,c_5 \rangle
\quad
V_7^\prime = \langle b_5,b_8,b_9,c_2,c_3,c_4,c_5 \rangle
\\
V_8 = \langle b_2,b_3,b_5,b_6,b_8,c_2,c_4,c_5 \rangle
\quad
V_8^\prime = \langle b_2,b_5,b_6,b_8,b_9,c_2,c_4,c_5 \rangle 
\quad
V_8^{\prime\prime} = \langle b_2,b_5,b_6,b_8,c_2,c_3,c_4,c_5 \rangle
\\
V_9 = \langle b_2,b_3,b_5,b_6,b_8,b_9,c_2,c_4,c_5 \rangle
\quad
V_9^\prime = \langle b_2,b_5,b_6,b_8,b_9,c_2,c_3,c_4,c_5 \rangle 
\\
V_{11} = \langle b_2,b_3,b_4,b_5,b_6,b_8,c_1,\dots,c_5 \rangle
\quad
V_{11}^\prime = \langle b_2,b_3,b_5,b_6,b_8,b_9,c_1,\dots,c_5 \rangle
\\
V_{11}^{\prime\prime} = 
\langle b_2,b_3,b_5,b_6,b_8,b_9,c_2,c_3,c_4,c_5,d \rangle
\\
V_{12} = \langle b_2,b_3,b_4,b_5,b_6,b_8,b_9,c_1,\dots,c_5 \rangle 
\end{gather*}
}
\end{proposition}

\begin{proof}
By Proposition \ref{prop-sandr}(\ref{it-sk}), the sandwich subalgebra $S$ of 
$\mathscr L$ coincides with $\langle c_2,c_4,c_5 \rangle$. Starting from this, 
rewrite the specified subspaces in invariant terms:

\begin{itemize}

\item $\langle c_5 \rangle = \set{x\in S}{\dim [\mathscr L,x] \leq 3}$;

\item 
$\langle c_4,c_5 \rangle$ is the subspace (actually, the abelian subalgebra) 
linearly spanned by elements $x\in S$ such that $\dim [\mathscr L,x] \leq 4$;

\item $\langle c_4 \rangle = [\mathscr L,c_5] \cap \langle c_4,c_5 \rangle$;

\item $V_4 = [\mathscr L,c_4]$;

\item $\langle c_2 \rangle = S \cap V_4$;

\item $V_5 = [\mathscr L,c_2]$;

\item $\langle b_8,c_2 \rangle = V_4 \cap V_5$;

\item $V_7^\prime = [\mathscr L,\langle b_8,c_2 \rangle]$

\item $V_9^\prime = \set{x\in \mathscr L}{[x,S] = 0}$;

\item 
$\langle b_5,c_2 \rangle = 
[\mathscr L,c_5] \cap V_4 \cap [V_9^\prime,V_9^\prime]$;

\item $V_8^\prime = \set{x \in V_9^\prime}{x^{[2]} \in \mathscr L}$;

\item $V_8^{\prime\prime} = [\mathscr L, \langle b_5,c_2 \rangle]$;

\item 
Note that $V_8^\prime$ is a subalgebra of $\mathscr L$ with the center $S$. The
subspace $V_7$ (actually, a subalgebra) coincides with the set of elements 
$x \in V_8^\prime$ such that the induced adjoint map 
$\ad x: V_8^\prime/S \to V_8^\prime/S$ has rank $\le 1$;

\item $V_6^\prime = V_7 \cap V_8^{\prime\prime}$;

\item $V_6^{\prime\prime} = V_7 \cap V_7^\prime$;

\item $V_{11}^\prime = \set{x \in \mathscr L}{[x,S] \subseteq S}$;

\item $V_{11}^{\prime\prime} = C_{\mathscr L}(c_4)$;

\item $V_9 = [V_{11}^{\prime\prime},V_{11}^{\prime\prime}]$;

\item 
$V_{12} = 
\set{x \in \mathscr L}{[x,V_{11}^\prime] \subseteq V_{11}^\prime + Kx}$;

\item $V_{11} = [V_{12},V_{12}]$;

\item $V_8 = V_9 \cap V_{11}$;

\item 
$V_6$ is a linear span of elements $x^{[2]}$, where $x \in V_8$;

\end{itemize}

All this is verified by straightforward computations.
\end{proof}

Many more $\Aut(\mathscr L)$-invariant subspaces of $\mathscr L$ can be produced
in a similar fashion, we confine here ourselves only to those which will be 
needed in the sequel.

\subsection{No other automorphisms}

\begin{theorem}\label{l-b3}
Assume that any quadratic equation with coefficients in the ground field $K$ has
a solution in $K$. Then $\Aut(\mathscr L) = \Aut_0(\mathscr L)$.
\end{theorem}

\begin{proof}
Let $\varphi$ be an automorphism of $\mathscr L$. Our strategy is consecutively
``twist'' $\varphi$ by taking compositions with various automorphisms from 
$\Aut_0(\mathscr L)$, and eventually arrive to the conclusion 
$\varphi = \id_{\mathscr L}$.

By Proposition \ref{flag}, $\varphi(b_3)$ lies in $V_6$ and does not lie in 
$V_4 + \langle c_4 \rangle$, i.e., is of the form
$$
\varphi(b_3) = 
\lambda_2 b_2 + \lambda_3 b_3 + \lambda_5 b_5 + \lambda_8 b_8 + 
\mu_2 c_2 + \mu_4 c_4 ,
$$
where $\lambda_3 \ne 0$. As $b_3$ is toral, $\varphi(b_3)$ is toral, and the 
equality $\varphi(b_3) = \varphi(b_3)^{[2]}$ is equivalent to the following 
quadratic system:
\begin{align*}
&\lambda_2 = \lambda_2\lambda_3     \\
&\lambda_3 = \lambda_3^2            \\
&\lambda_5 = \lambda_3\lambda_5     \\
&\lambda_8 = \lambda_3\lambda_8     \\
&\mu_2 = \lambda_3\mu_2             \\
&\mu_4 = \lambda_2^2 .
\end{align*}

Consequently,
\begin{equation}\label{eq-x1}
\varphi(b_3) = \lambda_2 b_2 + b_3 + \lambda_5 b_5 + \lambda_8 b_8 + \mu_2 c_2 
+ \lambda_2^2 c_4 .
\end{equation}

Assume $\lambda_2 \ne 0$. Applying to both sides of the equality (\ref{eq-x1}) 
the automorphism
$$
\exp(\mu_2 c_2) \>\circ\> 
\Theta(\alpha) \>\circ\> 
\Phi(\lambda_5) \>\circ\> 
\Delta\Big(\frac{1}{\sqrt{\lambda_2}}\Big) ,
$$
where $\alpha$ satisfies the quadratic equation 
$\alpha^2 + \alpha + \lambda_8 = 0$ (this is the only place were we need the 
assumption that any quadratic equation with coefficients in $K$ has a solution),
we may assume
$$
\varphi(b_3) = b_2 + b_3 + c_4 .
$$

The automorphism $\varphi$ maps the subalgebra $C_{\mathscr L}(b_3)$ to the 
subalgebra $C_{\mathscr L}(b_2 + b_3 + c_4)$. We have
$$
C_{\mathscr L}(b_3) = \langle b_3, b_6, b_9, c_3, c_4, c_5, d \rangle
$$
and
$$
C_{\mathscr L}(b_2 + b_3 + c_4) = 
\langle b_2 + b_3 + c_4, b_5 + b_6, b_8 + b_9, c_2 + c_3, c_4, c_5, d \rangle .
$$

By Proposition \ref{flag}, $\varphi(b_6)$ lies in $V_6^\prime$, thus does not 
contain terms with $b_2 + b_3 + c_4$, $b_8 + b_9$, $c_2 + c_3$, and $d$. 
Similarly, $\varphi(b_9)$ lies in $V_6^{\prime\prime}$, thus does not contain terms with 
$b_2 + b_3 + c_4$, $b_5 + b_6$, $c_2 + c_3$, and $d$. Therefore, we can write
\begin{alignat*}{4}
&\varphi(b_6) &\>=\>& \alpha_6(b_5 + b_6) + \alpha_4 c_4 + \alpha_5 c_5  
\\
&\varphi(b_9) &\>=\>& \beta_9(b_8 + b_9) + \beta_4 c_4 + \beta_5 c_5     
\\
&\varphi(d) &\>=\>& \gamma_3(b_2 + b_3 + c_4) + \gamma_6(b_5 + b_6) 
+ \gamma_9(b_8 + b_9) + \delta_3(c_2 + c_3) + \delta_4 c_4 + \delta_5 c_5 
+ \gamma d 
\end{alignat*}
for certain $\alpha_i, \beta_i, \gamma_i, \delta_i, \gamma \in K$. Then the 
equalities $[\varphi(b_6),\varphi(d)] = \varphi(b_3)$ and 
$[\varphi(b_9),\varphi(d)] = \varphi(b_6)$ are equivalent to
\begin{gather}
\alpha_6\gamma = 1                    \label{eq-1} \\
\alpha_6\delta_3 + \alpha_5\gamma = 1 \label{eq-2}
\end{gather}
and
\begin{gather}
\beta_9\gamma = \alpha_6   \label{eq-bga}   \\
\beta_5\gamma = \alpha_4      \\
\beta_9\delta_3 = \alpha_5 \label{eq-last}
\end{gather}
respectively. It is easy to see that the system (\ref{eq-1})--(\ref{eq-last}) is
contradictory: for example, multiplying (\ref{eq-last}) by $\alpha_6$, and 
taking into account (\ref{eq-2}), we get 
$\beta_9 + \alpha_5\beta_9\gamma = \alpha_5\alpha_6$, what, in its turn, 
together with (\ref{eq-bga}) gives $\beta_9 = 0$, hence 
$\alpha_5 = \alpha_6 = 0$, what contradicts (\ref{eq-2}).

Therefore $\lambda_2 = 0$. Applying to both sides of (\ref{eq-x1}) the 
automorphism
$$
\Psi(\lambda_5) \>\circ\> 
\exp(\mu_2 c_2) \>\circ\> 
\Theta\big(\sqrt{\lambda_8}\big) ,
$$
we may assume $\varphi(b_3) = b_3$. Consequently, the eigenspace 
$\langle b_1, b_2, b_4, b_5, b_7, b_8, c_1, c_2 \rangle$ corresponding to the 
eigenvalue $1$ of $\ad b_3$, is invariant under $\varphi$, and we may write
$$
\varphi(b_1) = \lambda_1 b_1 + \lambda_2 b_2 + \lambda_4 b_4 + \lambda_5 b_5 + 
               \lambda_7 b_7 + \lambda_8 b_8 + \mu_1 c_1 + \mu_2 c_2 .
$$

Applying to both sides of this equality the automorphism
$$
\exp(\lambda_5 c_5)       \>\circ\> 
\exp(\lambda_8 c_3^{[2]}) \>\circ\> 
\exp(\lambda_2 c_4)       \>\circ\> 
\Phi(\lambda_4)            \>\circ\>
\Delta\big(\sqrt{\lambda_1}\big) ,
$$
we may assume $\lambda_1 = 1$ and 
$\lambda_2 = \lambda_4 = \lambda_5 = \lambda_8 = 0$.

Then we have
$$
0 = \varphi(b_1^{[4]}) = \varphi(b_1)^{[4]} = 
\mu_1^2 b_4^{[2]} + \lambda_7^4 b_7^{[2]} + \mu_2^2 c_3^{[2]} + 
\text{ (terms lying in $\mathscr L$)} ,
$$
what implies $\lambda_7 = \mu_1 = \mu_2 = 0$, and $\varphi(b_1) = b_1$. Further:
\begin{itemize}

\item $\varphi(c_4) \in \langle c_4 \rangle$ by Proposition \ref{flag};

\item 
$\varphi(b_2) = [\varphi(b_1),\varphi(c_4)] \in [b_1,\langle c_4 \rangle] = 
 \langle b_2 \rangle$;

\item 
$\varphi(c_2) \in \langle c_2 \rangle$ by Proposition \ref{flag};

\item 
$\varphi(c_3) = [\varphi(b_1),\varphi(c_2)] \in [b_1, \langle c_2 \rangle] =
 \langle c_3 \rangle$;

\item
$\varphi(c_1) = [\varphi(b_1),\varphi(c_3)] \in [b_1, \langle c_3 \rangle] =
 \langle c_1 \rangle$;

\item 
$\varphi(c_5) \in \langle c_5 \rangle$ by Proposition \ref{flag};

\item 
$\varphi(b_5) = [\varphi(b_1),\varphi(c_5)] \in [b_1, \langle c_5 \rangle] =
 \langle b_5 \rangle$;

\item
$\varphi(b_6) = [\varphi(b_1),\varphi(b_5)] \in [b_1, \langle b_5 \rangle] =
 \langle b_6 \rangle$;

\item 
$\varphi(d) = [\varphi(b_1),\varphi(c_1)] \in [b_1, \langle c_1 \rangle] =
 \langle d \rangle$;

\item
$\varphi(b_7) = [\varphi(c_1),\varphi(d)] \in 
[\langle c_1 \rangle, \langle d \rangle] = \langle b_7 \rangle$;

\item
$\varphi(b_4) = [\varphi(b_7),\varphi(d)] \in 
[\langle b_7 \rangle, \langle d \rangle] = \langle b_4 \rangle$;

\item
$\varphi(b_8) = [\varphi(c_2),\varphi(d)] \in 
[\langle c_2 \rangle, \langle d \rangle] = \langle b_8 \rangle$;

\item
$\varphi(b_9) = [\varphi(c_3),\varphi(d)] \in 
[\langle c_3 \rangle, \langle d \rangle] = \langle b_9 \rangle$.

\end{itemize}

Hence $\varphi$ is a diagonal automorphism, and by Lemma \ref{lemma-diag} we 
have $\varphi = \Delta(\lambda)$, where $\lambda^{-2} = 1$. But then 
$\lambda = 1$ and $\varphi = \id_{\mathscr L}$.
\end{proof}

\section{Gradings}

Having a supply of automorphisms of $\mathscr L$ at hand, and using the known 
correspondence between automorphisms and group gradings (formulated in full 
generality in the language of affine group schemes -- see, for example, 
\cite[Proposition 1.36]{elduque-kochetov}), we may try to construct group 
gradings of $\mathscr L$.

In practice, this is achieved by extending the ground field $K$ to a suitable 
commutative ring $R$ (not necessary a field -- we are dealing with the group 
scheme $R \to \Aut_R(\mathscr L \otimes_K R)$), extending an automorphism 
$\varphi$ of $\mathscr L$ to the automorphism $\overline\varphi$ of 
$\overline{\mathscr L} = \mathscr L \otimes_K R$ via 
$\overline\varphi(x \otimes r) = \varphi(x) \otimes r$, and considering 
eigenspaces 
$
\overline{\mathscr L}_\lambda = 
\set{x \in \overline{\mathscr L}}{\overline\varphi(x) = \lambda x} .
$
For a suitable choice of $\varphi$ and $R$, and a suitable homomorphism $\chi$ 
from the group generated by all eigenvalues $\lambda \in R$ to a group $G$, the
eigenspaces will be ``rational'': 
$\overline{\mathscr L}_\lambda = \mathscr L_{\chi(\lambda)} \otimes_K R$. The 
ensued grading 
$\mathscr L = \bigoplus_{\chi(\lambda)} \mathscr L_{\chi(\lambda)}$ will be a 
grading of $\mathscr L$ by $G$, even in the case when the eigenvalues $\lambda$
do not necessarily belong to the ground field.

For example, the diagonal automorphism $\Delta(\lambda)$ for the ``generic'' 
value of $\lambda$ (or, what is the same, for $\lambda \in K$ such that the 
order of $\lambda$ in the multiplicative group $K^*$ is $>7$) produces a 
$\mathbb Z$-grading
\begin{alignat}{10}\label{eq-z}
&\hspace{2pt}\scriptstyle{-2}& 
&\hspace{7pt}\scriptstyle{-1}& 
&\hspace{10pt}\scriptstyle{0}&
&\hspace{11pt}\scriptstyle{1}&
&\hspace{10pt}\scriptstyle{2}&
&\hspace{11pt}\scriptstyle{3}&
&\hspace{10pt}\scriptstyle{4}&
&\hspace{9pt}\scriptstyle{5}
\notag \\
\mathscr L =\>
\langle &b_1&     \rangle \oplus 
\langle &b_4,d&   \rangle \oplus 
\langle &b_3,b_7& \rangle \oplus 
\langle &b_6,c_1& \rangle \oplus 
\langle &b_2,b_9& \rangle \oplus 
\langle &b_5,c_3& \rangle \oplus 
\langle &b_8,c_4& \rangle \oplus 
\langle &c_2,c_5& \rangle .
\end{alignat}

Specializing the automorphism $\Delta(\lambda)$ to the cases $\lambda^n = 1$, 
$n \le 7$, we obtain a $\mathbb Z / n\mathbb Z$-grading of $\mathscr L$. 

As the product of any two elements from the standard basis is either zero, or
again an element from the standard basis, the decomposition of $\mathscr L$ into
one-dimensional subspaces spanned by the basis element is a grading. This is a 
group grading, and its universal group (\cite[\S 1.2]{elduque-kochetov}) is 
isomorphic to the (additive) abelian group with generators $x,y$ (corresponding
to elements $b_1$, $b_4$ respectively), and the relation $2x = 4y$, what is 
easily seen to be isomorphic to the direct sum 
$\mathbb Z \oplus \mathbb Z / 2\mathbb Z$. On the other hand, the thin 
decomposition exhibited in \S \ref{ss-thin} is a 
$(\mathbb Z / 2\mathbb Z)^4$-grading.

In general, it seems to be a difficult task to classify all group gradings of 
the Skryabin algebra. Even to determine whether the automorphisms of order $2$
(exponential, $\Phi$, and $\Theta$) lead to $\mathbb Z / 2\mathbb Z$-gradings 
seems to be far from trivial. Perhaps, it could be approached with the method
of \cite{krutov-leb}.

\section{Comparison of the Skryabin algebra with algebras from the Eick list. 
A bit more numerology}\label{ss-eick}

In \cite{eick}, a computer-generated list of simple Lie algebras over $\GF(2)$ 
of dimension $\le 20$ is presented. The Skryabin algebra (defined over $\GF(2)$)
is not isomorphic to any of the $15$-dimensional algebras in the list.

One way to see this, using data from \cite{eick}, is to look at automorphisms, 
either at the group of exponential automorphisms $\Exp(\mathscr L)$, or at the 
whole group $\Aut(\mathscr L)$. As the ground field is $\GF(2)$, 
$\Exp(\mathscr L)$ is isomorphic to the additive group 
$\GF(2) \oplus \GF(2) \oplus \GF(2) \oplus \GF(2)$, thus having order $16$, what
is different from all the $15$-dimensional algebras in the list (including the 
new ones, number 7 and 8, dubbed by us here as $Eick_7$ and $Eick_8$) except for
the non-alternating Hamiltonian algebra $P(2,1,1)$ (number 4). The order of 
$\Aut(\mathscr L)$ is $2^7 = 128$ (over $\GF(2)$, there are no nontrivial 
diagonal automorphisms), what is, again, different from all the 
$15$-dimensional algebras in the list.

Another way to distinguish between all these algebras, is to repeat for them the
same pedestrian, but informative computations concerning tori and the sandwich
subalgebra, as in \S \ref{ss-num} and \S \ref{sec-misc}. The following table 
accumulates some information about the $15$-dimensional central simple algebras
from the Eick list, whose derivation algebra is $19$-dimensional (and coincides
in all the cases with the $2$-envelope). Here $TR$ denotes the absolute toral
rank of the algebra, $N_1$ and $N_m$ denote, respectively, the number of toral 
elements, and of tori of the maximal dimension $TR$ in the $2$-envelope, and $S$
is the dimension of the sandwich subalgebra.

$$
\begin{tabular}{|l|r|r|r|r|}
\hline
\multicolumn{1}{|c|}{\small{algebra}} & \multicolumn{1}{|c|}{\small{$TR$}} & 
\multicolumn{1}{|c|}{\small{$N_1$}} & \multicolumn{1}{|c|}{\small{$N_m$}} & 
\small{$S$}
\\ \hline
$W(4)$     & 2 & 256 &  1,536 & 10
\\ \hline
$P(2,1,1)$ & 4 & 448 & 43,680 & 1
\\ \hline
$P(3,1)$   & 3 & 384 & 10,752 & 5
\\ \hline
$P(2,2)$   & 4 & 384 & 13,440 & 3
\\ \hline
$Eick_7$   & 4 & 464 & 87,360 & 1
\\ \hline
$Eick_8$   & 4 & 464 & 67,200 & 1
\\ \hline
\end{tabular}
$$

\medskip

Interestingly enough, all of these algebras which are of absolute toral rank 
$4$, i.e., $P(2,1,1)$, $P(2,2)$, $Eick_7$, and $Eick_8$, also admit thin 
decompositions with respect to a lot (conjecturally with respect to all) of the $4$-dimensional tori.

\section{Open questions}\label{s-quest}

{\bf 1.} Is it true that two toral elements $h$ and $h^\prime$ in $\mathscr L$ 
are conjugate with respect to the automorphism group, if and only if 
$C_{\mathscr L}(h) \simeq C_{\mathscr L}(h^\prime)$ ?

\medskip

{\bf 2.} Describe all gradings of the Skryabin algebra.

\medskip

{\bf 3.} \emph{Conjecture}. 
Any simple Lie algebra of dimension $>3$ over a field of characteristic $2$ 
admitting a thin decomposition, has:
\\
a) a proper simple graded subalgebra (with respect to this decomposition);
\\
b) a graded subalgebra isomorphic either to $W$ or to $H$.

\medskip

{\bf 4.} Classify simple finite-dimensional Lie algebras over an algebraically 
closed field of characteristic $2$, admitting a thin decomposition.

\medskip

{\bf 5.} 
Classify simple finite-dimensional $\mathbb Z$-graded Lie algebras over an 
algebraically closed field of characteristic $2$, such that all homogeneous 
components are of dimension $<3$. (Note that the Skryabin algebra belongs to this class, due to 
(\ref{eq-z})).

\section*{Acknowledgements}

Thanks are due to Bettina Eick, Dimitry Leites, and Alexander Premet for useful
remarks and interesting discussions. 
Alexander Grishkov was supported by FAPESP, grant number 2018/23690-6 and 
CNPq, grant number 307824/2016-0;
Marina Rasskazova during the work on this paper was visiting the University of
S\~ao Paulo, and was supported by FAPESP, grant number 2018/11292-6.

\end{document}